\documentclass[11pt,a4paper]{article}
\usepackage{amsmath,amsthm}
\usepackage{url}
\usepackage{times}
\usepackage{graphicx}
\usepackage{color}
\usepackage{amssymb}

\newtheorem{theorem}{Theorem}
\newtheorem{remark}[theorem]{Remark}
\newtheorem{corollary}[theorem]{Corollary}
\newtheorem{lemma}[theorem]{Lemma}

\newtheorem{proposition}[theorem]{Proposition}

\newcommand{\comment}[1]{}

\newcommand{\qhypergeomphi}[5]{\mbox{$
		_#1 \phi_#2 \left(
		\begin{array}{c}
			\multicolumn{1}{c}{\begin{array}{c} #3
			\end{array}}\\[1mm]
			\multicolumn{1}{c}{\begin{array}{c} #4
		\end{array}}\end{array}
		; \displaystyle{#5}\right) $} }

\usepackage{amssymb}
\usepackage{amsfonts}
\usepackage{times}

\definecolor{gold}{rgb}{0.83, 0.69, 0.22}

\definecolor{greenncs}{rgb}{0.0,0.62,0.42}

\begin{document}
	
	\title{Indeterminate moment problem associated with continuous dual $q$-Hahn polynomials}
	\author{K. JORDAAN$^a$ and M. KENFACK NANGHO$^b$\\
		\\
		$^a$ Department of Decision Sciences,\\ University of South Africa, PO Box 392, Pretoria, 0003, South Africa\\{\tt jordakh@unisa.ac.za}\\ +27822999511\\	
		$^b$Department of Mathematics and Computer Science,\\ University of Dschang, PO Box 67, Dschang, Cameroon\\ {\tt maurice.kenfack@univ-dschang.org}\\+23773616764}
	\maketitle
	
	\date{}
	
	\begin{abstract}
		\noindent We study a limiting case of the Askey-Wilson polynomials when one of the parameters goes to infinity, namely continuous dual $q$-Hahn polynomials when $q>1$. Solutions to the associated indeterminate moment problem by general theory are found and an orthogonality relation is established.
	\end{abstract}
	
	\begin{itemize}
		\item[]\textbf{Keywords}: Moment problem; continuous dual q-Hahn polynomials; divided-difference operators
		\item[]\textbf{Mathematics Subject Classification (2020)}: Primary 30E05; Secondary 33D45
	\end{itemize}

	\section{Introduction}
	One of the classical questions in functional analysis, related to orthogonal polynomials, is the moment problem: for a sequence of real numbers $\{\mu_n\}_{n\geq 0}$, is there a positive Borel measure $\mu$ with {supp$(\mu)$}, the support of $\mu$, in $\mathbb{R}$ such that
	
	\begin{equation}\label{e1}\mu_n=\int x^nd\mu, \;n\in\{0,1,2,...\}?\end{equation}
	When such a measure exists, if the measure is unique, then the moment problem is \emph{determinate} otherwise the moment problem is called an \emph{indeterminate moment problem}. When the sequence $\{\mu_n\}_{n\geq 0}$ is positive, that is for all $n\in\{0,1,2,..\}$ the Hankel determinant
	\begin{equation*}
		D_n=\left|\begin{array}{cccc}
			\mu_0&\mu_1&...&\mu_n\\
			\mu_1&\mu_2&...&\mu_{n+1}\\
			....&....&....&....\\
			\mu_n&\mu_{n+1}&...&\mu_{2n}
		\end{array}
		\right|
	\end{equation*}
	is positive, there exists a positive Borel 
	measure $\mu$, supported on $\mathbb{R}$, such that (\ref{e1}) is satisfied (cf. \cite{AK}). Moreover, the family of polynomials 
	
	\begin{equation*}
		P_n(x)=\frac{1}{\sqrt{D_nD_{n-1}}}\left|\begin{array}{cccc}
			\mu_0&\mu_1&...&\mu_n\\
			\mu_1&\mu_2&...&\mu_{n+1}\\
			....&....&....&....\\
			\mu_{n-1}&\mu_{n}&...&\mu_{2n-1}\\
			1&x&...&x^n
		\end{array}
		\right|,\qquad n=1,2\dots,	\end{equation*}
	{with $P_0(x)=1$, is} {orthonormal} with respect to the measure $\mu$ on its support, and the sequence $\{P_n\}_{n=0}^{\infty}$ satisfies the three-term recurrence relation 
	\begin{equation}\label{e2}xP_n(x)=b_{n}P_{n+1}(x)+a_nP_n(x)+b_{n-1}P_{n-1}(x),\end{equation}
	where \begin{equation*}b_{n}=\frac{\sqrt{D_{n-1}D_{n+1}}}{D_n},\qquad a_n=\int xP_n(x)^2d\mu(x).\end{equation*} The initial conditions are $P_{-1}(x)=0$ and $P_0(x)=1$. Observing from the above relation that the leading coefficient of $P_n$, $n=0,1,2,...$, is $\displaystyle{\frac{1}{b_0b_1...b_{n-1}}}$, the monic polynomial $\displaystyle{p_n(x)={b_0b_1...b_{n-1}}P_n(x)}$ satisfies the three-term recurrence relation
	\begin{equation}\label{e3}xp_n(x)=p_{n+1}(x)+c_np_n(x)+\lambda_{n}p_{n-1}(x),\end{equation}
	with  $c_n=a_n {\in\mathbb{R}}$ and $\lambda_n=b_{n-1}^2>0$, where $a_n$ and $b_{n-1}$ are the constants appearing in (\ref{e2}). 
	The three-term recurrence relations (\ref{e2}) and (\ref{e3}) both provide useful information about the moment problem. For example, in \cite[Thm 2]{chihara1989}, Chihara proved that, under the assumption
	\[\lim_{n\rightarrow \infty}c_n=\infty\;\;{\rm and}\;\; \lim_{n\rightarrow \infty}\frac{\lambda_{n+1}}{c_nc_{n+1}}=L<\frac{1}{4},\] 
	the moment problem is indeterminate when
	\begin{equation}\label{IC}\displaystyle{\liminf_{n\rightarrow\infty} c_n^\frac{1}{n}>\frac{1+\sqrt{1-4L}}{1-\sqrt{1-4L}}}\end{equation}
	and determinate when
	 the opposite (strict) inequality holds. Specifically, when $c_n=f_nq^{-n}$ with $0 < q < 1$ and where $\{f_n\}_{n\geq0}$ is both bounded and bounded away from $0$, the moment problem is determinate when $$L~\displaystyle{<\frac{q}{(1+q)^2}}$$
	and indeterminate when $$L~\displaystyle{>\frac{q}{(1+q)^2}}.$$ 
	
	\medskip \noindent When the moment problem is indeterminate, there are infinitely many measures satisfying (\ref{e1}) (cf. \cite{Nevan-1922}). Moreover, these measures can be described by means of the Stieltjes integral
	\begin{equation}\label{S}
		\int_{{\text{supp}}(\mu)} \frac{d\mu(t)}{z-t}= \frac{A(z)\varphi(z)- {C(z)}}{ {B(z)}\varphi(z)-D(z)},\;\;z\in \mathbb{C}\setminus {{\text{supp}}(\mu)},
	\end{equation} 
	where $\varphi$ belongs to the space of Pick functions augmented with the point $\infty$. The entire functions $A, B, C$ and $D$ are uniform limits, on compact subsets of $\mathbb{C}$, {as $n\to\infty$}, of 
	\begin{eqnarray*}
		A_n(z)&=&b_n(Q_n(0)Q_{n+1}(z)-Q_{n+1}(0)Q_n(z)),\;\;
		B_n(z)=b_n(Q_n(0)P_{n+1}(z)-Q_{n+1}(0)P_n(z)),\\
		C_n(z)&=&b_n(P_n(0)Q_{n+1}(z)-P_{n+1}(0)Q_n(z)),\;\;
		D_n(z)=b_n(P_n(0)P_{n+1}(z)- {P_{n+1}(0)P_n(z)}),
	\end{eqnarray*}
	where $\{Q_n\}_{n=0}^{\infty}$ is the first associated orthogonal polynomial system defined by the three-term recurrence relation (\ref{e2}) with the initial condition $Q_0(x)=0$ and $Q_1(x)=1$. In cases where it is possible to express the so-called Nevanlinna matrix $$\left({\begin{array}{cc}
				A & C \\
				C & D \\
		\end{array} } \right)$$ in terms of known functions, formula \eqref{S} is a useful tool for determining explicit measures that solve the indeterminate moment problems (cf. \cite{ChiharaIsmail1993,Jacob,Jacob-2005,Ismailbook}). A useful collection of examples of indeterminate moment problems can be found in the recent publication \cite[\S 11.3]{BergChristiansen}. If $\varphi$ is of the form $\varphi (z)=t,\; t\in \mathbb{R}\cup \{\infty\}$, the space of polynomials $\mathbb{C}[x]$ is dense in $L^2(\mathbb{R},\mu)$. 
	
	\medskip \noindent An interesting example of an indeterminate moment problem on the real line is the moment problem corresponding to continuous $q$-Hermite polynomials \cite[(14.26.1)]{KSL} for $q>1$. Askey was the first to give an explicit weight function for continuous $q$-Hermite polynomials when $q>1$ (cf. \cite{Askey1989}). {Ismail and Masson studied these polynomials extensively and introduced the name $q^{-1}$-Hermite polynomials in \cite{Mourad1994}. Berg and Ismail \cite{BergIsmail} showed that continuous $q$-Hermite polynomials are fundamental in the hierarchy of classical $q$-orthogonal polynomials and can be used to systematically build other $q$-orthogonal polynomials by attaching generating functions to measures. They used the solution of the $q^{-1}$-Hermite moment problem to derive a family of explicit solutions to a special case of the Al-Salam-Chihara moment problem for $q>1$. The indeterminate moment problem for Al-Salam-Chihara polynomials when $q>1$ was first discussed by Askey and Ismail in \cite{AskeyIsmail1984}.} In \cite{Jacob-2005}, Christiansen and Ismail gave more solutions for the $q^{-1}$-Hermite moment problem and also discussed the moment problem for a symmetric case of the Al-Salam-Chihara polynomials. The indeterminant moment problems associated with polynomials in the Askey scheme were classified and investigated in \cite{Jacob}. Christiansen and Koelink (cf. \cite{JacobErik2008}) provided  an alternative derivation of the N-extremal measures for the continuous $q^{-1}$-Hermite and  Al-Salam-Chihara polynomials. More recent contributions to the indeterminate Hamburger moment problem associated with Al-Salam-Chihara polynomials are due to Groenevelt (cf. \cite{Groen}) and Ismail (cf. \cite{Mourad2020}). In \cite{Mourad2020}, new infinite families of orthogonality measures were provided for $q^{-1}$-Hermite polynomials, $q$-Laguerre polynomials and Stieltjes-Wigert polynomials.
	
	\medskip \noindent The moment problem for continuous dual $q$-Hahn polynomials \cite[(14.3.1)]{KSL} when $q>1$ was first pointed out by Askey and Wilson in \cite[p. 31-32]{Askey-1985} as the limiting case, letting one parameter approach infinity, of the Askey-Wilson polynomials \cite[(1.15)]{Askey-1985} \cite[(14.1.1)]{KSL}
	
	\begin{equation}\label{Askey}
		\frac{a^np_n(x;a,b,c,d|q)}{(ab,ac,ad;q)_n}=  \qhypergeomphi{4}{3}{q^{-n},\,abcdq^{n-1},\,ae^{-i\theta},ae^{i\theta}}{ab,\,ac,\,ad}{q,q},\;x=\cos\theta.
	\end{equation}
	Here the $q$-shifted factorials are given by \begin{equation*}(a;q)_0=1, \qquad (a;q)_k=\displaystyle{\prod_{j=0}^{k-1}\left(1-aq^j\right)} \text{ for } k=1,2,\dots \text{ or } \infty,\end{equation*}
	while the multiple $q$-shifted factorials are defined by
	\begin{equation*}(a_1,\dots,a_i;q)_k=\displaystyle{\prod_{j=1}^{i}(a_j;q)_k}\end{equation*} and 
	\begin{equation}\label{e3a}
		\qhypergeomphi{{s+1}}{s}{a_1,\dots ,a_{s+1}}{b_1,\dots,b_s}{q,z}=\sum_{k=0}^{\infty}\frac{(a_1,\dots,a_{s+1};q)_k}{(b_1,\dots,b_s;q)_k}\frac{z^k}{(q;q)_k}.
	\end{equation}
	
	\medskip \noindent The limiting case for Askey-Wilson polynomials  for $q>0$ and $q \neq 1$, as $d$ tends to infinity, is given by (cf. \cite[Proposition 3.6(i)]{MK2019})
	\begin{equation*}\lim_{d\rightarrow\infty}\frac{p_n(x;a,b,c,d|q)}{(ad;q)_n}=(bc)^nq^{n(n-1)}p_n(x;a^{-1},b^{-1},c^{-1}|q^{-1}),~~ n=0,1,2,...,\end{equation*}
	where $p_n(x;a,b,c|q)$ denotes continuous dual $q$-Hahn polynomials (cf. \cite[(14.3.1)]{KSL}). The orthogonality relation for $p_n(x;a,b,c|q)$ when $0<q<1$ can be found in \cite[(14.3.2)]{KSL}.
	
	\medskip \noindent In the monograph \cite{Jacob}, Christiansen proved that the moment problem associated with continuous dual $q$-Hahn polynomials  when $q>1$ is indeterminate,  
	and obtained solutions for the special case when $c=0$ and $b=-a$. The moment problem for the general case of continuous dual $q$-Hahn polynomials when $q>1$, for parameter values $a,~b$ and $c$ such that $ab$, $ac$ and $bc \in(-1,0)$, was considered by Koelink and Stokman in \cite{KStokman} where they used spectral analysis of a $q$-difference operator to obtain an explicit measure as well as a natural orthogonal basis of the complement of the polynomials in the corresponding weighted $L^2$ space. The solution has an absolutely continuous as well as a discrete part.
	
	\medskip \noindent 
	In this paper, we use our recent results in \cite{MK2019} to derive {certain orthogonality} measures {of the} continuous dual $q$-Hahn polynomials for $q>1$ and parameter values $a,~b$ and $c$ with $ab$, $ac$ and $bc \notin[-1,0]$ {obtained as a limit case of the Askey-Wilson polynomials}. Note that, since the polynomial in (\ref{Askey}) is symmetric in $a,\,b,\,c$ and $d$, it suffices to study the case when $d \to\infty$. We state and prove an orthogonality relation satisfied by the continuous dual $q$-Hahn polynomials. {The measures we obtained}  {solve the indeterminate moment problem associated with continuous dual $q$-Hahn polynomials by general theory. That is, we found measures, not necessarily positive, that can be used to compute moments associated with this family of orthogonal polynomials.}
	
	\section{Continuous dual \boldmath{$q^{-1}$}-Hahn polynomials} 
	
	Assume throughout that $0<q<1$.  The monic polynomials $\displaystyle{\tfrac{1}{2^n}p_n(x;a^{-1},b^{-1},c^{-1}|q^{-1})}$ satisfy the three-term recurrence relation (\ref{e3}) with (cf. \cite[(3.30)]{MK2019}) 
	\begin{equation}\label{e4a}c_n={\frac {(ab+ac+bc){q}^{n}+{q}^{n+1}-q-1}{2abc {q}^{2n}}}\; {\rm and}\; 
		{\lambda}_n={\frac { \left(1- {q}^{n} \right)  \left( 1-bc{q}^{n-1} \right)
					\left(1- ac{q}^{n-1} \right)  \left(1- ab{q}^{n-1} \right) }{4{a}^{2}{b}^{2}{c}^{
						2}  {q}^{4n-3}}}.\end{equation}
	
	\medskip \noindent For $0<q<1$, $\displaystyle{\lim_{n\rightarrow\infty} c_n= \infty}$ and  \begin{align*}L=\lim_{n\rightarrow\infty}\frac{\lambda_{n+1}}{c_n c_{n+1}}&=\lim_{n\rightarrow\infty}\frac{q(1-q^{n+1})(1-bcq^n)(1-acq^n)(1-abq^n)}{((ab+ac+bc)q^n+q^{n+1}-q-1)((ab+ac+bc)q^{n+1}+q^{n+2}-q-1)} \\
		&=\frac{q}{(q+1)^2}<\frac{1}{4}.\end{align*}  Moreover, since $c_n=f_nq^{-2n}$ with $$f_n=\frac{(ab+ac+bd)q^n+q^{n+1}-q-1}{2abc}$$ bounded, it follows that $$\displaystyle\lim_{n\rightarrow \infty}c_n^{\frac{1}{n}}=\frac{1}{q^2}>\frac{1}{q}=\frac{1+\sqrt{1-4L}}{1-\sqrt{1-4L}}.$$ Therefore (\ref{IC}) holds and  the moment problem associated with the polynomials $\displaystyle{\tfrac{1}{2^n}p_n(x;a^{-1},b^{-1},c^{-1}|q^{-1})}$ is indeterminate. 
	
	
	\medskip \noindent Letting $a,\,b$ and $c$ go to infinity in (\ref{e4a}), we obtain $c_n=0$ and $\lambda_n=\frac14\left(1-\frac{1}{q^n}\right)$, and therefore (\ref{e3}) reads as
	\[p_{n+1}(x)=xp_n(x)-\frac{q^{-n}(q^n-1)}{4}p_{n-1}(x).\]
	The corresponding monic polynomial system (cf. \cite[Prop. 3.6 (iv)]{MK2019}), known as $q^{-1}$-Hermite polynomials (cf. \cite{Mourad1994}, \cite[p. 533]{Ismailbook}), is orthogonal on the imaginary axis (cf. \cite{Askey1989}). 
	
	\medskip \noindent In order to keep the orthogonality on the real line when one or more parameters tends to infinity in (\ref{Askey}), we follow Askey \cite{Askey1989} in using the change of variable $x\rightarrow ix$, as well as the change of parameters, $(a,b,c)\rightarrow (-ia,-ib,-ic)$, or, equivalently, $(a^{-1},b^{-1},c^{-1})\rightarrow (ia^{-1},ib^{-1},ic^{-1})$, to obtain
	\begin{subequations}\label{rec}
		\begin{equation}
			q_{n+1}(x;a,b,c|q)=(x-\tilde{c}_n)q_n(x;a,b,c|q)-\tilde{\lambda}_nq_{n-1}(x;a,b,c|q)
		\end{equation}
		with  
		\begin{align}
			\tilde{c}_n=&{\frac {(ab+ac+bc){q}^{n}-{q}^{n+1}+q+1}{2abc  {
						q}^{2n}}}, \\
			\tilde{\lambda}_n=&-\frac{1}{4}\left( 1-\frac{1}{q^n} \right)  \left( 1+{
				\frac {1}{bc{q}^{n-1}}} \right)  \left( 1+{\frac {1}{ac{q}^{n-1}}}
			\right)  \left( 1+{\frac {1}{ab{q}^{n-1}}} \right),
		\end{align}
	\end{subequations}where continuous dual $q^{-1}$-Hahn polynomials $q_n(x;a,b,c|q)$ are defined as
	\begin{align}\label{e4b}
		q_n(x;a,b,c|q)=&\nonumber \frac{(-i)^n}{2^n}p_n(ix;ia^{-1},ib^{-1},ic^{-1}|q^{-1})\\=&(-\frac{a}{2})^n\left(-\frac{1}{ab},-\frac{1}{ac};\frac{1}{q}\right)_n\qhypergeomphi{3}{2}{q^{n},\,-\frac{e^{u}}{a},\frac{e^{-u}}{a}}{-\frac{1}{ab},\,-\frac{1}{ac}}{\frac{1}{q},\frac{1}{q}}, ~~x=\sinh(u).\end{align}

	\medskip \noindent
	The analogue of the Askey-Wilson operator for the parametrisation $x=\sinh(u)$, introduced by Ismail (cf. \cite{Mourad-1993}), is given by  
	\begin{equation}\label{e0.1}
		\mathcal{D}_qf(x)=\frac{\breve{f}(q^{\frac{1}{2}}e^u)-\breve{f}(q^{-\frac{1}{2}}e^{u})}{(q^{\frac{1}{2}}-q^{-\frac{1}{2}})\cosh u}\end{equation} with $$~~\breve f(e^u)=f\left(\frac{e^u-e^{-u}}{2}\right)=f(x).$$ The divided-difference operator \eqref{e0.1}, as well as the averaging operator \cite[(21.6.3)]{Ismailbook}
	\begin{equation*}\label{aveop}\mathcal{S}_qf(x)=\frac{\breve{f}(q^{\frac{1}{2}}e^u)+\breve{f}(q^{-\frac{1}{2}}e^{u})}{2},
	\end{equation*}
	will play a fundamental role in the sequel.
	
	\begin{proposition}\label{thm1} The action of the divided-difference operator $\mathcal{D}_q$ defined in \eqref{e0.1} on continuous dual  $q^{-1}$-Hahn polynomial $q_n(x;a,b,c|q)$, $n=1,2,3,...$ is given by  
		\begin{equation}\label{e5}
			\mathcal{D}_q q_n\left(x;a,b,c|q\right)=\gamma_nq_{n-1}\left(x;aq^\frac{1}{2},bq^\frac{1}{2},cq^\frac{1}{2}|q\right),\,\gamma_n=\frac{q^{\frac n2}-q^{-\frac n2}}{q^\frac{1}{2}-q^{-\frac{1}{2}}}.
	\end{equation} \end{proposition}
	\comment{\begin{proof} 
			Use $(\ref{e4b})$ as well as (\ref{e3a}) to expand $q_n(x;a,b,c|q)$ in terms of $v_{k}(u,a;q)=\left(-\frac{e^u}{a},\,\frac{e^{-u}}{a};\frac{1}{q}\right)_{k},$ $k=1,2...$, apply the operator $\mathcal{D}_q$ to both sides and take into account the relation $$\mathcal{D}_qv_{k}(u,a;q)= \frac {2 \left({q}^{k} -1 \right)}{a{q}^{k-1}  \left( q-1 \right) }
			v_{k-1}(u,aq^{\frac12};q),\; k\geq 1,$$ obtained by using (\ref{e0.1}) with $f(x)=\left(-\frac{e^u}{a},\,\frac{e^{-u}}{a};\frac{1}{q}\right)_k,\;x=\sinh (u)$, to obtain, after some computation, (\ref{e5}).\end{proof}}
	\begin{proof} 
	By definition, the action of the operator $\mathcal{D}_q$ on $v_{k}(u,a;q)=\left(-\tfrac{e^u}{a},\,\frac{e^{-u}}{a};\frac{1}{q}\right)_{k}$, $k=1,2...$ is \begin{align*}(q^{\frac 12}-q^{-\frac 12})&\left(\frac{e^u+e^{-u}}{2}\right)\mathcal{D}_qv_k(u,a;q)\\=&\left(-\frac{e^uq^{\frac 12}}{a},\frac{e^{-u}}{aq^{\frac 12}};\frac 1q\right)_k-\left(-\frac{e^u}{aq^{\frac 12}},\frac{e^{-u}q^{\frac 12}}{a};\frac 1q\right)_k\\=&\left(1+\frac{e^uq^{\frac 12}}{a}\right)\left(1+\frac{e^u}{aq^{\frac 12}}\right)...\left(1+\frac{e^u}{aq^{\frac{2k-3}{2}}}\right)\left(1-\frac{e^{-u}}{aq^{\frac 12}}\right)\left(1-\frac{e^{-u}}{aq^{\frac 32}}\right)...\left(1-\frac{e^{-u}}{aq^{\frac{2k-1}{2}}}\right)\\&-\left(1+\frac{e^u}{aq^{\frac 12}}\right)\left(1+\frac{e^u}{aq^{\frac 32}}\right)...\left(1+\frac{e^u}{aq^{\frac{2k-1}{2}}}\right)\left(1-\frac{e^{-u}q^{\frac 12}}{a}\right)\left(1-\frac{e^{-u}}{aq^{\frac 12}}\right)...\left(1-\frac{e^{-u}}{aq^{\frac{2k-3}{2}}}\right)\\=&
				\left(1+\frac{e^u}{aq^{\frac 12}}\right)...\left(1+\frac{e^u}{aq^{\frac{2k-3}{2}}}\right)\left(1-\frac{e^{-u}}{aq^{\frac 12}}\right)...\left(1-\frac{e^{-u}}{aq^{\frac{2k-3}{2}}}\right)\\&\times\left[\left(1+\frac{e^uq^{\frac 12}}{a}\right)\left(1-\frac{e^{-u}}{aq^{\frac{2k-1}{2}}}\right)-\left(1+\frac{e^{u}}{aq^{\frac{2k-1}{2}}}\right)\left(1-\frac{e^{-u}q^{\frac 12}}{a}\right)\right]\\=&v_{k-1}(u,aq^{\frac12};q)\frac 1a\left[e^uq^{\frac 12}-\frac{e^{-u}}{q^{\frac{2k-1}{2}}}-\frac{e^{u}}{q^{\frac{2k-1}{2}}}+e^{-u}q^{\frac 12}\right]\\=&v_{k-1}(u,aq^{\frac12};q)\frac 1a\left[(e^u+e^{-u})(q^{\frac 12}-q^{-\frac{2k-1}{2}})\right],\end{align*} which yiels the relation $$\mathcal{D}_qv_{k}(u,a;q)= \frac {2 \left({q}^{k} -1 \right)}{a{q}^{k-1}  \left( q-1 \right) }
		v_{k-1}(u,aq^{\frac12};q),\; k\geq 1.$$ Use $(\ref{e4b})$ as well as (\ref{e3a}) to expand $q_n(x;a,b,c|q)$ in terms of $v_{k}(u,a;q)$ and apply the operator $\mathcal{D}_q$ to obtain, after some computation, (\ref{e5}).
		\comment{\begin{align*}
				\mathcal{D}_qq_n(x;a,b,c|q)&=\left(-\frac a2\right)^n\left(\frac{-1}{ab},\frac{-1}{ac};\frac 1 q\right)_n\,\sum_{k=1}^{n}\frac{v_{k-1}(u,aq^{\frac 12};q)\left(q^n;\frac 1q\right)_k\left(\frac 1q\right)^k}{\left(\frac{-1}{ab},\frac{-1}{ac},\frac 1q;\frac 1 q\right)_k}\\&=\left(-\frac a2\right)^{n-1}\frac{1}{1-q}\left(\frac{-1}{abq},\frac{-1}{acq};\frac 1 q\right)_{n-1}\,\sum_{k=1}^{n}\frac{\left(\frac{-e^{u}}{aq^{\frac12}},\frac{e^{-u}}{aq^{\frac 12}};\frac{1}{q}\right)_{k-1}\left(q^n;\frac 1q\right)_k}{\left(\frac{-1}{abq},\frac{-1}{acq},\frac 1q;\frac 1 q\right)_{k-1}}\\&=
				\left(\frac{-a}{2}\right)^{n-1}\frac{1-q^n}{1-q}\left(\frac{-1}{abq},\frac{-1}{acq};\frac 1 q\right)_{n-1}\,\sum_{k=0}^{n-1}\frac{\left(\frac{-e^{u}}{aq^{\frac12}},\frac{e^{-u}}{aq^{\frac 12}};\frac{1}{q}\right)_{k}\left(q^n;\frac 1q\right)_k}{\left(\frac{-1}{abq},\frac{-1}{acq},\frac 1q;\frac 1 q\right)_{k}}
			\end{align*} which yields (\ref{e5}) after some computation.}\end{proof}
\begin{lemma}\label{Lemma1} Let $x(s)=\tfrac 12(q^s-q^{-s})$, $q^s=e^u$, $\alpha_n=\tfrac 12(q^{\frac n2}+q^{-\frac n2})$ and $$\gamma_n=\frac{q^{\frac n2}-q^{- \frac n2}}{q^\frac{1}{2}-q^{-\frac{1}{2}}}\quad \text{for} \quad n=0,1,\dots.$$ Then polynomial solutions $P_n(x)$ of degree exactly $n$ of the Sturm-Liouville type equation
			\begin{equation*}
				\phi(x)\mathcal{D}_{q}^2y(x)+\psi(x)\mathcal{S}_{q}\mathcal{D}_{q} y(x)+\lambda y(x)=0,\,\label{e1.10.1}
			\end{equation*}
			where $\phi(x)=\phi_2x^2+\phi_1x+\phi_0$ and $\psi(x)=\psi_1x+\psi_0$ are polynomials of degree at most two and one, can be expanded as
			\begin{equation*}
				P_n(x)=\sum_{k=0}^{n} d_k\prod_{j=0}^{k-1}[x(s)-x(\mu+j)],
			\end{equation*}
			where $\mu$ is a complex number such that $\sigma(x(\mu))=0$
			with \begin{equation*}
				\label{sigma} \sigma(x(s))=\phi(x(s))-\frac{x(s+\frac 12)-x(s-\frac12)}{ 2}\psi (x(s))
			\end{equation*} and $d_k$ is solution to the first order recurrence relation
			\begin{align}\label{3.4}\nonumber
				&\left(\gamma_{k}\gamma_{k+1}\left(\phi_2\left(x(\mu+k)+x(\mu)\right)+\phi_1-\frac{\psi_1(x(\mu+\frac12)-x(\mu-\frac 12))}{ 2}\right)+\alpha_k\gamma_{k+1}\psi(x(\mu+k))\right)d_{k+1}\\&+\left(\lambda+\gamma_k\gamma_{k-1}\phi_2+\gamma_k\alpha_{k-1}\psi_1\right)d_k
				=0
			\end{align} with $\lambda=-\gamma_{n}\gamma_{n-1}\phi_2-\gamma_n\alpha_{n-1}
			{\psi_1}$.
		\end{lemma}
		\begin{proof}See \cite[Lemma 3.1]{MK2018} for a proof,  observing that
			\begin{subequations}\label{change}	\begin{align}
					\mathcal{D}_qf(x)&=\frac{f(x(s+\frac{1}{2}))-f(x(s-\frac{1}{2}))}{x(s+\frac{1}{2})-x(s-\frac{1}{2})}=\mathbb{D}_xf(x(s)),\\
					\mathcal{S}_qf(x)&=\frac{f(x(s+\frac{1}{2}))+f(x(s-\frac{1}{2}))}{2}=\mathbb{S}_xf(x(s)),
			\end{align}\end{subequations} when $x(s)=\tfrac 12(q^s-q^{-s})$ with $q^s=e^u.$   \end{proof}
	\begin{proposition}\label{thm2}Continuous dual  $q^{-1}$-Hahn polynomials $q_n(x;a,b,c|q)$, $n=0,1,\,2,\dots$ solve the Sturm-Liouville type equation
		\begin{equation}\label{e6a}
			\phi(x)\mathcal{D}_q^2y(x)+\psi(x)\mathcal{S}_q\mathcal{D}_qy(x)+\lambda y(x)=0,
		\end{equation}
		with
		\begin{subequations}\label{e6}
			\begin{align}
				{\phi(x)=}&{2\,{x}^{2}+ \left( {\frac {1}{abc}}-\frac{1}{a}-\frac{1}{b}-\frac{1}{c}
					\right) x+{\frac {1}{ab}}+{\frac {1}{ac}}+{\frac {1}{bc}}+1},\\
				{\psi(x)=}&{{\frac {4\sqrt {q}}{q-1}}x-2\,{\frac {\sqrt {q} \left( bc+ac+ab+1
							\right) }{abc \left( q-1 \right) }}},\\
				\lambda=&-4{\frac {\sqrt {q} \left( {q}^{n}-1 \right) }{ \left( q-1 \right) ^{2}}}. 
			\end{align}
	\end{subequations}\end{proposition}
	\begin{proof}
		Observing that, for $x(s)=\tfrac 12(q^s-q^{-s})$, $q^s=e^u$ and $q^{\eta}=-a$, 
		\[v_{k}(u,a;q)=\left(\frac{2}{a}\right)^kq^{-\frac{k(k-1)}{2}}\prod_{j=0}^{k-1}[x(s)-x(\eta+j)]\]
		it follows that
		\[q_n(x;a,b,c|q)=\sum_{k=0}^{n} d_{k}\prod_{j=0}^{ k-1}[x(s)-x(\eta+j)], \]
		where \begin{equation}d_{k}=\frac{(-a)^n}{2^n}\left(-\frac{1}{ab},-\frac{1}{ac};\frac{1}{q}\right)_n\frac{(q^n;\frac 1q)_k(\frac 2aq^{-\frac{k+1}{2}})^k}{(-\frac{1}{ab},-\frac{1}{ac};\frac{1}{q})_k(\frac{1}{q};\frac{1}{q})_k}, \quad k=0,1,2,....\label{dk}\end{equation} Assuming that $q_n(x,a,b,c|q),n=0,1,\,2,\dots$ satisfies (\ref{e6a}) with {$\phi(x)=\phi_2x^2+\phi_1x+\phi_0$} and {$\psi(x)=\psi_1x+\psi_0$}, the use of \eqref{3.4} with $d_k$ given by \eqref{dk}, leads to
		\[\sum_{j=1}^{4}H_j(\phi_2,\, \phi_1,\,\phi_0,\psi_1,\,\psi_0,\,\lambda,\,q^n)(q^k)^j=0,\]
		where $H_j$, $j=1,2,3,4,$ is a linear combination of $\phi_2,\, \phi_1,\,\phi_0,\psi_1,\,\psi_0,$ and $\lambda$. Solving the system of linear equations $H_j(\phi_2,\, \phi_1,\,\phi_0,\psi_1,\,\psi_0,\,\lambda,\,q^n)=0$ in terms of the coefficients $\phi_2,\phi_1$, $\psi_1$ and $\psi_0$, we obtain
		\begin{eqnarray*}
			&& \phi_{{2}}=-{\frac { \left( q-1 \right) ^{2}\lambda}{2\sqrt {q} \left( {q
					}^{n}-1 \right) }}
			,\;\phi_{{1}}={\frac { \left( q-1 \right) ^{2} \left( ab+ac+bc-1 \right) 
					\lambda}{4abc \left( {q}^{n}-1 \right) \sqrt {q}}}
			,\\
			&&\psi_{{1}}
			=-{\frac { \left( q-1 \right) \lambda}{{q}^{n}-1}},\;\psi_{0}={\frac { \left( q-1 \right)  \left( ab+ac+bc+1 \right) \lambda}{2
					abc \left( {q}^{n}-1 \right) }}. 	  	
		\end{eqnarray*}
		Since these coefficients do not depend on $n$, taking  $$\lambda=-4{\frac {\sqrt {q} \left( {q}^{n}-1 \right) }{ \left( q-1 \right) ^{2}}},$$ we get $\phi_2,\, \phi_1,\,\psi_1$ and $\psi_0$ as given in the theorem. Substituting $\phi(x)$ and $\psi(x)$ into (\ref{e6a}) with $y(x)=q_n(x,a,b,c|q)$, and taking $n=2$ we derive $\phi_0$.

	\end{proof}
	\begin{proposition} \label{thm3}If $\{q_n(x;a,b,c|q)\}_{n=0}^{\infty}$ is orthogonal with respect to a weight function $w(x)$ then \newline $\{\mathcal{D}_q^2q_n(x;a,b,c|q)\}_{n=2}^{\infty}$ is orthogonal with respect to $\pi(x)\,w(x)$, where $\pi(x)$ is the polynomial defined by
		\begin{equation*}
			\pi(x)=\frac{8}{abc}(x+\frac{a-a^{-1}}{2})(x+\frac{b-b^{-1}}{2})(x+\frac{c-c^{-1}}{2}).
		\end{equation*}
	\end{proposition}
	
	\begin{proof}
		Since $q_n(x;a,b,c|q),\, n=0,1,2,\dots,$ satisfies (\ref{e6a}), it follows from \eqref{change} and \cite[Thm 4.2]{MK2018} that \newline $\{\mathcal{D}_q^2q_n(x;a,b,c|q)\}_{n\geq 2}$ is orthogonal with respect to $\pi(x)\,w(x)$, where (cf. \cite[Remark 4.7]{MK2018}) 
		\begin{equation}\pi(x)=\phi(x)^2-U_2(x)\psi(x)^2\label{Pi}\end{equation}
		and $$U_2(x)=\left(\frac{x(s+\frac{1}{2})-x(s-\frac{1}{2})}{2}\right)^2, \quad x(s)=\frac{q^s-q^{-s}}{2}, q^s=e^u$$ is the polynomial given in \cite[p.5]{MK2018}. Observing that $e^u=x+\sqrt{1+x^2}$ and $e^{-u}=\sqrt{1+x^2}-x$, we see that
			\begin{align}
				U_2(x)&=\frac{1}{16}(e^u+e^{-u})^2(q^{\frac 12}-q^{-\frac12})^2\nonumber \\
				&=\frac{1}{4q}(1+x^2)(q-1)^2. \label{U}
		\end{align} Substituting \eqref{U} and the expresssions for $\phi$ and $\psi$, given in \eqref{e6}, into \eqref{Pi}, we obtain the result.

	\end{proof}
	\section{{Orthogonality measures} for continuous \boldmath{dual $q^{-1}$}-Hahn polynomials}
	
	\begin{theorem}\label{theorem4} Let $a,b$ and $c$ be real numbers such that $ab,ac, bc \notin\{-q^{-k},k=1,2,...\}$ and $ab$, $ac$, $bc$ $\notin [-1,0]$. Then the sequence of continuous dual  $q^{-1}$-Hahn polynomials $\left\{q_n(x;a,b,c|q)\right\}_{n}$ defined in (\ref{e4b}) is orthogonal with respect to the distribution $d\mu(x)= w(x;a ,b ,c )dx$, where
		\begin{equation}\label{e4c}
			w(x;a ,b ,c )=N(x;a ,b ,c )w(x),\; x=\sinh(u)
		\end{equation}
		with $$N(x;a ,b ,c )=\left(-\frac{q}{a}e^u,\frac{q}{a}e^{-u},-\frac{q}{b}e^u,\frac{q}{b}e^{-u},-\frac{q}{c}e^u,\frac{q}{c}e^{-u};q\right)_{\infty}$$ and $w(x)$ a $q^{-1}$-Hermite weight. 
	\end{theorem}
	\begin{proof}
		The continuous dual $q^{-1}$-Hahn polynomials satisfy the three-term recurrence relation \eqref{rec}. Since  $ab,ac, bc \notin\{-q^{-k},k=1,2,...\}$ and $ab$, $ac$, $bc$ $\notin [-1,0]$, $\tilde{\lambda}_n>0$ for all $n\in\mathbb{N}$, it follows from the spectral theorem for monic orthogonal polynomials that there exists a positive measure $\mu$ on the real line such that these polynomials are monic orthogonal polynomials satisfying 
		\begin{equation*}\int_{\mathbb{R} } q_n\left(x; a,  b, c| q\right)q_m\left(x;a, b,c|q\right)d\mu(x)=k_n\delta_{m,n}, \qquad m,n=0,1,2,\dots .\end{equation*} 
		Since $\lim\limits_{n\rightarrow\infty}\tilde{\lambda}_n=\infty$, the measure $\mu$ has infinite, unbounded support. We look for $d\mu(x)$ of the form $d\mu(x)= w(x;a ,b ,c )dx$, for a function depending continuously on the parameters $a,\,b$ and $c$. It follows from (\ref{e5}) that we have \begin{equation}\label{ref}\mathcal{D}_q^2q_n(x;a ,b ,c |q )=\gamma_n\gamma_{n-1}q_{n-2}(x;a q ,b q ,c q |q ).\end{equation} Therefore, the sequence $\left\{\mathcal{D}_q^2q_n(x;a ,b ,c |q )\right\}_{n=2}^{\infty}$ is orthogonal with respect to $w(x;aq,bq,cq)$. 
		{ Let us prove that
		\begin{eqnarray}\label{WR}
			w(x;aq,bq,cq)=(1+\frac{e^u}{a})(1-\frac{e^{-u}}{a})(1+\frac{e^u}{b})(1-\frac{e^{-u}}{b})(1+\frac{e^u}{c})(1-\frac{e^{-u}}{c})w(x;a,b,c).
		\end{eqnarray}
		Let $n$ be an integer. The formal Fourier expansion of $$ \frac{w(x;aq,bq,cq)}{w(x;a,b,c)}\mathcal{D}_q^2q_n(x;a ,b ,c |q )$$ in the system $\{q_k\}_{k=0}^{\infty}$ is \[\frac{w(x;aq,bq,cq)}{w(x;a,b,c)}\mathcal{D}_q^2q_n(x;a ,b ,c |q )=\sum_{k=0}^{\infty}a_{n,k}q_k(x;a ,b ,c |q ) \]
		with
		\[ a_{n,k}\int_{-\infty}^{\infty}w(x;a,b,c)\left(q_k(x;a ,b ,c |q )\right)^2dx=\int_{-\infty}^{\infty}w(x;aq,bq,cq)\mathcal{D}_q^2q_n(x;a ,b ,c |q )q_k(x;a ,b ,c |q )dx.\]
		Since $\left\{\mathcal{D}_q^2q_n(x;a ,b ,c |q )\right\}_{n=2}^{\infty}$ is orthogonal with respect to $\pi (x)w(x;a,b,c)$ {by Proposition \ref{thm3}}, and $\pi$ is a polynomial of degree at most four, there exist constants $b_{n,n+j},\;j\in \{-2,-1,0,1,2\}$ such that ({cf.} \cite[Prop 4.4]{MK2018})
		\[q_k(x;a ,b ,c |q )=\sum_{j=-2}^{2}b_{k,k+j}\mathcal{D}_q^2q_{k+j}(x;a ,b ,c |q ).\]
		Therefore 
			\begin{eqnarray*}
			\lefteqn{a_{n,k}\int_{-\infty}^{\infty}w(x;a,b,c)q_k(x;a ,b ,c |q )^2dx}&&\\
			&&=\sum_{j=-2}^{2}b_{k,k+j}\int_{-\infty}^{\infty}w(x;aq,bq,cq)\mathcal{D}_q^2q_n(x;a ,b ,c |q )\mathcal{D}_q^2q_{k+j}(x;a ,b ,c |q )dx\\
			&&=0, \;\text{when}\; k+j\neq n, j=-2,-1,0,1,2.
			\end{eqnarray*}
		Hence $a_{n,k}=0$ for $k\notin\{n-2,n-1,n, n+1, n+2\}$ and 
			\begin{equation*}\label{SR}\frac{w(x;aq,bq,cq)}{w(x;a,b,c)}\mathcal{D}_q^2q_n(x;a ,b ,c |q )=\sum_{k=n-2}^{n+2}a_{n,k}q_k(x;a ,b ,c |q ).\end{equation*}
	From \eqref{ref}, for $n=2$, we have
	\[ \frac{w(x;aq,bq,cq)}{w(x;a,b,c)}\gamma_2\gamma_1=
		\sum_{k=0}^{4}a_{2,k}q_k(x;a ,b ,c |q ).\]
	This implies that $\displaystyle \frac{w(x;aq,bq,cq)}{w(x;a,b,c)}$ is a polynomial of degree at most four and hence there exist a polynomial, $Q(x)$, of degree at most four, such that 	$w(x;aq,bq,cq)=Q(x)w(x;a,b,c)$.				 	
		Combining with   \cite[Thm 4.2]{MK2018} and \cite[Remark 4.7]{MK2018} we have
		\begin{equation}\label{Q}Q(x)=\phi(x)^2-U_2(x)\psi(x)^2.\end{equation}  Substituting \eqref{U} and the expresssions for $\phi$ and $\psi$, given in \eqref{e6}, into \eqref{Q}, we obtain the result in (\ref{WR}).} 
		
	\noindent Finally, substituting $(a,b,c)$ by $(aq^{-1},bq^{-1},cq^{-1})$ in \eqref{WR} and then iterating the  relation obtained, yields
		\begin{eqnarray*}
			\lefteqn{w(x;a,b,c)}&\\&=\prod_{j=1}^{n}(1+\frac{q^je^u}{a})(1-\frac{q^je^{-u}}{a})(1+\frac{q^je^u}{b})(1-\frac{q^je^{-u}}{b})(1+\frac{q^je^u}{c})(1-\frac{q^je^{-u}}{c})w(x;aq^{-n},bq^{-n},cq^{-n}).
		\end{eqnarray*}
		Now letting $n$ tend to $\infty$ and using the definition of $q$-shifted factorial, the latter relation becomes 
		\[w(x;a ,b ,c )=\left(-\frac{q}{a}e^u,\frac{q}{a}e^{-u},-\frac{q}{b}e^u,\frac{q}{b}e^{-u},-\frac{q}{c}e^u,\frac{q}{c}e^{-u};q\right)_{\infty}\times \lim_{n\to\infty}w(x;aq^{-n},bq^{-n},cq^{-n}).\]
		Since $q_n(x;a,b,c|q)$ tends to $q^{-1}$-Hermite as all the parameters $a,b$ and $c$ go to $\infty$, $$\lim_{n\to\infty}w(x;aq^{-n},bq^{-n},cq^{-n}), \,0<q<1$$ is a continuous $q^{-1}$-Hermite weight.
		
	\end{proof}
	
	\medskip \noindent We are now ready to prove the orthogonality relation for continuous dual $q^{-1}$-Hahn polynomials.
	
	\begin{theorem}\label{thm4}
		Let $a,b$ and $c$ be real numbers such that $ab,ac, bc \notin\{-q^{-k},k=1,2,...\}$ and $ab$, $ac$, $bc$ $\notin [-1,0]$. Then the sequence of continuous dual  $q^{-1}$-Hahn polynomials $\{q_n(x;a ,b ,c |q )\}_{n}$ satisfies the orthogonality relation 
		\begin{equation}\label{or}
			\int_{-\infty}^{\infty}q_n\left(x; a, b, c|q\right)q_m\left(x; a, b, c|q\right)w(x;a,b,c)dx=k_n\delta_{m,n},
		\end{equation}
		where
		\begin{equation*}
			k_n=4^{-n}q^{-\frac{n(n+1)}{2}}\left(-\frac{q}{abq^n},\,-\frac{q}{acq^n},\,-\frac{q}{bcq^n};q\right)_{\infty}\left(q^n;\frac 1q\right)_n>0.
		\end{equation*}
	\end{theorem}
	
	\begin{proof}
		Let be $m$ and $n$ two non-negative integers such that $m\leq n$. We first evaluate the integral 
		\begin{equation}\label{e6b}\int_{-\infty}^{\infty}q_n(x;a ,b ,c |q )\left(\frac{-e^u}{b},\frac{e^{-u}}{b};\frac 1q\right)_mw(x;a ,b ,c )dx\end{equation}
		then derive the result by using the fact that $\{q_n(x;a ,b ,c |q)\}_n$ is symmetric in $(a,b,c)$ (i.e. permutations of $a,b$ and $c$ leave $q_n(x;a ,b ,c |q )$ unchanged).
		Let us evaluate (\ref{e6b}): 
		Observe that $e^u=x+\sqrt{1+x^2}$ and $e^{-u}=\sqrt{1+x^2}-x$ and take $t_1=\frac{q}{a}$, $t_2=\frac{q}{b}$, $t_3=\frac{q}{c}$ and $t_4=0$ into \cite[(3.8)]{Mourad1994} to obtain 
		\begin{equation}\label{e7}
			\int_{-\infty}^{\infty}\left(-\frac{q}{a}e^u,\frac{q}{a}e^{-u},-\frac{q}{b}e^u,\frac{q}{b}e^{-u},-\frac{q}{c}e^u,\frac{q}{c}e^{-u};q\right)_{\infty}w(x)dx=
			\left(-\frac{q}{ab},-\frac{q}{ac},-\frac{q}{bc};q\right)_{\infty},\,
			x=\frac{e^u-e^{-u}}{2}.\end{equation}
		  Using the relation \begin{equation}\label{e7a}\left(\alpha;\frac 1q\right)_l\left(\alpha\,q;q\right)_{\infty}=\left(\frac{\alpha\,q}{q^l};q\right)_{\infty},l=0,1,2,...,\end{equation}  we  have  from (\ref{e4c}), that the product 
		\[\left(\frac{-e^u}{a},\frac{e^{-u}}{a};\frac 1q\right)_k\left(\frac{-e^u}{b},\frac{e^{-u}}{b};\frac 1q\right)_mw(x;a ,b ,c )=N\left(x;aq^k,bq^m,c\right)w(x),\] where $k$ and $m$ are non-negative integers.
		Therefore, expanding $q_n(x;a ,b ,c |q )$ by use of  \eqref{e3a} and \eqref{e4b}, we obtain
		\begin{eqnarray*}
			\lefteqn{\int_{-\infty}^{\infty}q_n(x;a ,b ,c |q )\left(\frac{-e^u}{b},\frac{e^{-u}}{b};\frac 1q\right)_mw(x;a ,b ,c )dx}&&\\
			&&=\left(-\frac{a}{2}\right)^n\left(-\frac{1}{ab},-\frac{1}{ac};\frac{1}{q}\right)_n\sum_{k=0}^{n}\frac{(q^n;\frac{1}{q})_k(\frac 1q)^k}{(-\frac{1}{ab},-\frac{1}{ac};\frac{1}{q})_k(\frac{1}{q};\frac{1}{q})_k}\int_{-\infty}^{\infty}N\left(x;aq^k,bq^m,c\right)w(x)dx.\end{eqnarray*}
		Applying (\ref{e7}) with $(a,b,c)$ taken for $(aq^k, bq^m,c)$ and using (\ref{e7a}), first with $(\alpha,\,l)$ taken as $(-\frac{1}{ac},k)$ and then $(\alpha,\,l)$ taken as $(-\frac{1}{abq^m},k)$, we obtain
		\begin{equation*}
			\int_{-\infty}^{\infty}N\left(x;aq^k,bq^m,c\right)w(x)dx=\left(-\frac{1}{abq^{m}},-\frac{1}{ac};\frac{1}{q}\right)_{k}	\left(-\frac{q}{abq^{m}}, -\frac{q}{ac},-\frac{q}{bcq^m};q\right)_{\infty}.
		\end{equation*}
		Therefore
		\begin{eqnarray*}
			\lefteqn{\int_{-\infty}^{\infty}q_n(x;a ,b ,c |q )\left(\frac{-e^u}{b},\frac{e^{-u}}{b};\frac 1q\right)_mw(x;a ,b ,c )dx}&&\\
			&&=\left(-\frac{a}{2}\right)^n\left(-\frac{1}{ab},-\frac{1}{ac};\frac{1}{q}\right)_n\left(-\frac{q}{abq^m},\,-\frac{q}{bcq^m},\,-\frac{q}{ac};q\right)_{\infty}\qhypergeomphi{2}{1}{q^{n},\,-\frac{1}{abq^m}}{-\frac{1}{ab}}{\frac{1}{q},\frac{1}{q}}.
		\end{eqnarray*}
		Using the $q$-analogues of Vandermonde's formula \cite[(1.5.3)]{GR} with $(b,c,q)$  substituted by $(-\frac{1}{abq^m},-\frac{1}{ab},\frac 1q)$ we obtain
		\begin{equation*}
			\qhypergeomphi{2}{1}{q^{n},\,-\frac{1}{abq^m}}{-\frac{1}{ab}}{\frac{1}{q},\frac{1}{q}}=\frac{\left(q^m;\frac 1q\right)_n}{\left(-\frac{1}{ab};\frac 1q\right)_n}\left(-\frac{1}{abq^m}\right)^n.
		\end{equation*}
		Therefore taking into account (\ref{e7a}) with $\alpha=-\frac{1}{ac}$ and $l=n$, we obtain
		\begin{eqnarray}
			\lefteqn{\int_{-\infty}^{\infty}q_n(x;a ,b ,c |q )\left(\frac{-e^u}{b},\frac{e^{-u}}{b};\frac 1q\right)_mw(x;a ,b ,c )dx}\nonumber&&\\
			&&=\left(\frac{1}{2bq^m}\right)^n\left(-\frac{q}{abq^m},\,-\frac{q}{bcq^m},\,-\frac{q}{acq^n};q\right)_{\infty}\left(q^m;\frac 1q\right)_n.\label{e8a}
		\end{eqnarray}
		Since $q_m$ is symmetric in $(a,b,c)$, interchanging $a$ and $b$ into $q_m(x;a ,b ,c |q )$ and expanding using (\ref{e3a}) we obtain
		\[q_m(x;a ,b ,c |q )=\frac{\left(-\frac{b}{2}\right)^m\left(q^m;\frac{1}{q}\right)_mq^{-m}}
		{\left(\frac{1}{q};\frac{1}{q}\right)_m}\left(-\frac{e^u}{b},\frac{e^{-u}}{b};\frac{1}{q}\right)_m+\dots.\]
		 Taking into account the relation $\left(q^m;\frac{1}{q}\right)_m=(-1)^mq^{\frac{m(m+1)}{2}}
		 \left(\frac{1}{q};\frac{1}{q}\right)_m,m=0,1,2,\dots,$ we obtain after simplification  
		\[q_m(x;a ,b ,c |q )=\left(\frac{b}{2}\right)^mq^{\frac{m(m-1)}{2}}\left(-\frac{e^u}{b},\,\frac{e^{-u}}{b};\frac{1}{q}\right)_m+\dots.\]
	Combining the previous relation with (\ref{e8a}) we obtain \eqref{or}. Finally, since $ab,ac,bc\notin [-1,0]$ implies that $$\left(-\frac{q}{abq^n},\,-\frac{q}{acq^n},\,-\frac{q}{bcq^n};q\right)_{\infty}>0,$$ we have that $k_n>0$.
	
	\end{proof}

	\begin{remark} \hspace{ 1 cm}\\ 
			\begin{enumerate}
		\item {Since a continuous $q^{-1}$-Hermite weight is} 
		\cite[(21.6.13)]{Ismailbook}
		\begin{equation}\label{e9}
			w(x)=\frac{e^u}{\left(-e^{2u},-qe^{-2u};q\right)_{\infty}}, x=\sinh(u),
		\end{equation}  {the} orthogonality measure $d\mu= w(x;a ,b ,c )dx$ with
		\begin{equation}\label{e8}
		w(x;a ,b ,c )=\frac{e^u\left(-\frac{q}{a}e^u,\frac{q}{a}e^{-u},-\frac{q}{b}e^u,\frac{q}{b}e^{-u},-\frac{q}{c}e^u,\frac{q}{c}e^{-u};q\right)_{\infty}}{(-e^{2u},\,-qe^{-2u};q)_{\infty}},
		\end{equation}
		$x=\tfrac 12(e^u-e^{-u}),\,u\in \mathbb{R}$ { is a} { solution to the indeterminate moment problem for continuous dual $q^{-1}$-Hahn polynomials in} {a} {general sense.} { However, letting $c$ tend to $+\infty$ in (\ref{e8}) and substituting $(a^{-1}, b^{-1})$ by $(a, b)$ we obtain 
			\[w(x,a,b)=\frac{e^u\left(-{qa}e^u,{qa}e^{-u},-{qb}e^u,{q b}e^{-u};q\right)_{\infty}}{(-e^{2u},\,-qe^{-2u};q)_{\infty}}=\frac{1}{2}w_{ab}(x,a,b),\; x=\sinh(u),\]
			where $w_{ab}(x,a,b)$, for $b=\bar{a}$, $\Im{a\neq0}$}  is the weight function \cite[(3.9)]{Mourad2020} for {a} solution to the indeterminate moment problem associated with Al-Salam Chihara polynomials.
			\item Letting $(ia,ib,ic, e^u)=(t_0^{-1}, t_1^{-1},t_2^{-1}, i\gamma)$, with $i^2=-1$, in \eqref{e4b}, we obtain up to a multiplicative factor the polynomial in \cite[(9.8)]{KStokman}. The condition in Theorem \ref{theorem4} that $ab$, $ac$, $bc$ $\notin [-1,0]$  therefore corresponds to  $t_it_j<1$ in the notation used in \cite{KStokman}. Since the results in \cite{KStokman} hold for parameters $t_0$, $t_1$, $t_2$ with $t_i>0$ and  $t_it_j>1$ when $i\neq j$, our work {contributes to the general theory for the} indeterminate moment problem associated with continuous dual $q^{-1}$-Hahn polynomials for parameter values that have not been considered in \cite{KStokman}.
		\end{enumerate}	
	\end{remark}
    \section{Acknowledgements}
    KJ gratefully acknowledges the support of a Royal Society Newton Advanced Fellowship NAF$\backslash$R2$\backslash$180669. MKN thanks the Royal Society for supporting his research visit to the University of South Africa for the period 20 Sept-30 Oct 2019, during which the main results of this work were proved. 
	

\begin{thebibliography}{99}
		
		\bibitem{AK}N. I. Akhiezer, {\it The Classical Moment Problem and Some Related Questions in Analysis}, English translation (Oliver and Boyd, Edinburgh, 1965).
		
		\bibitem{Askey1989} R. A. Askey, Continuous $q$-Hermite polynomials when   $q   >  1$, in {\it $q$-Series and Partitions}, ed. D. Stanton, IMA Vol. Math. Appl., (Springer-Verlag, New York, 1989), 151--158.
		\bibitem{AskeyIsmail1984} R. A. Askey and M. E. H. Ismail, Recurrence relations, continued fractions and orthogonal polynomials, {\it Mem. Amer. Math. Soc.} {\bf 300} (1984). 
		\bibitem{Askey-1985} R. A. Askey and J. A. Wilson, Some basic hypergeometric orthogonal polynomials that generalize
		Jacobi polynomials, {\it Mem. Amer. Math. Soc.} {\bf 319} (1985).
	\bibitem{BergChristiansen} C. Berg and J. S. Christiansen, {\it The Moment Problem, pages 269-306} In: Encyclopedia of Special Functions. The Askey-Bateman Project Volume I. Univariate Orthogonal Polynomials. Editor M. E. H. Ismail. Cambridge University Press, Cambridge UK 2020.
	\bibitem{BergIsmail} C. Berg  and M. E. H. Ismail, $q$-Hermite polynomials and classical orthogonal polynomials, {\it Can. J. Math}, {\bf 48}(1) (1996) 43--63.
		\bibitem{chihara1989}  T. S. Chihara,
		Hamburger moment problems and orthogonal polynomials, {\it Trans. Amer. Math. Soc.} {\bf 315}(1) (1989) 189--203.
		\bibitem{ChiharaIsmail1993}T. S. Chihara and M. E. H. Ismail, Extremal measures for a system of orthogonal polynomials, {\it Constr. Approx.} {\bf 9}(1) (1993), 111–-119.
		\bibitem{Jacob} J. S. Christiansen, Indeterminate Moment Problems within the Askey-scheme, (University of Copenhagen, Denmark, 2004).
		\bibitem{Jacob-2005}J. S. Christiansen and M. E. H. Ismail, A moment problem and family of integral evaluations,  {\it Trans. Amer. Math. Soc.} {\bf 358} (2006) 4071--4097.
		\bibitem{JacobErik2008} J.S. Christiansen and E. Koelink, Self-Adjoint Difference Operators and Symmetric
		Al-Salam-Chihara Polynomials, {\it Constr. Approx.} {\bf 28} (2008) 199--218. 
		\bibitem{GR} G. Gasper and M. Rahman, {\it Basic hypergeometric series}, second edition, Encyclopedia of Mathematics and Its Applications, Vol. 96 (Cambridge University Press, Cambridge, 2004).
		\bibitem{Groen} W. Groenvelt, A solution to the Al-Salam-Chihara Moment Problem, in {\it Positivity and Noncommutative Analysis}, Buskes G. et al. (eds), Trends in Mathematics, (Birkh\"auser, Cham, 2019), 223--248.
		\bibitem{Mourad2020} M. E. H. Ismail, Solutions of the Al-Salam-Chihara and allied moment problems, {\it Anal. Appl.} {\bf 18}(2) (2020), 185--210.
		\bibitem{Mourad-1993} M. E. H. Ismail, Ladder operators for $q^{-1}$-Hermite polynomials, {\it C.R. Math. Rep. Acad. Sci. Canada.} {\bf 15}(6) (1993) 261--266.
		\bibitem{Ismailbook} M. E. H. Ismail, {\it Classical and Quantum Orthogonal Polynomials in One Variable}, Encyclopedia of Mathematics and its Applications, vol. 98 (Cambridge University press, Cambridge, 2005).
		\bibitem{Mourad1994} M. E. H. Ismail and D. Masson, $q$-Hermite polynomials, biorthogonal rational functions, and $q$-beta integrals. {\it Trans. Amer. Math. Soc}. {\bf 346} (1994) 63-116.
		
		\bibitem{MK2019} M. Kenfack-Nangho and K. Jordaan, A characterization of Askey-Wilson polynomials, {\it Proc. Amer. Math. Soc.} {\bf 147} (2019) 2465--2480.
		
		\bibitem{MK2018}M. Kenfack-Nangho and K. Jordaan, Structure Relations of Classical Orthogonal Polynomials in the Quadratic and $q$-Quadratic Variable, {\it SIGMA} {\bf 14} (2018) 126, 26 pp.
		
		
		\bibitem{KSL}R. Koekoek, P. A. Lesky, and R. F. Swarttouw, Hypergeometric orthogonal polynomials and their $q$-analogues, {\it Springer Monographs in Mathematics}, (Springer-Verlag, Berlin, 2010).
		
		\bibitem{KStokman} E. Koelink and J. V. Stokman, The big q-Jacobi function transform, {\it Constr. Approx.} {\bf 19}(2) (2003) 191-235.
		
		
		\bibitem{Nevan-1922} R. H. Nevanlinna, Asymptotische Entwicklungen beschr${\rm \ddot{a}}$nkter Funktionen und das Stieltjessche Momentenproblem, {\it Ann. Acad. Sci. Fenn. (A)}, {\bf 18} (1922), no. 5, 1--52.
	\end{thebibliography}
\end{document}